\documentclass{amsart}
\usepackage[T1]{fontenc}
\usepackage{lmodern}
\usepackage{microtype}
\usepackage{amssymb, amsmath}
\usepackage[shortlabels, inline]{enumitem}
\usepackage{xcolor}
\usepackage{centernot}
\usepackage[colorlinks,
linkcolor={red!50!black}, citecolor={green!40!black}, urlcolor={blue!50!black}
]{hyperref}

\newtheoremstyle{slanted}%
  {}{}%
  {\slshape}{}%
  {\bfseries}{.}%
  { }{}
\theoremstyle{slanted}
\newtheorem{thm}{Theorem}[section]
\newtheorem{cor}[thm]{Corollary}
\newtheorem{lem}[thm]{Lemma}
\newtheorem{prop}[thm]{Proposition}
\newtheorem*{prop*}{Proposition}
\newtheorem*{fact*}{Fact}
\newtheorem*{claim*}{Claim}

\theoremstyle{definition}
\newtheorem{defi}[thm]{Definition}
\newtheorem{exa}[thm]{Example}
\theoremstyle{remark}
\newtheorem*{rem*}{Remark}

\newcommand{\0}{\mathclap{\phantom{\emptyset}}0}
\IfFileExists{emptysetiszero}{}{\renewcommand{\0}{\emptyset}}

\makeatletter
\newcommand{\aru}{\@ifstar\@saru\@aru}
\newcommand{\@saru}[1]{\mathop{\left(\exists #1\right)}}
\newcommand{\@aru}[1]{\mathop{\exists #1}}
\newcommand{\zenbu}{\@ifstar\@szenbu\@zenbu}
\newcommand{\@szenbu}[1]{\mathop{\left(\forall #1\right)}}
\newcommand{\@zenbu}[1]{\mathop{\forall #1}}
\makeatother
\newcommand{\caru}[1]{{\left\langle#1\right\rangle}}
\newcommand{\czenbu}[1]{{\left[#1\right]}}
\newcommand{\inv}{{-1}}
\newcommand{\Fr}{Fra\"iss\'e}
\DeclareMathOperator{\dom}{dom}
\DeclareMathOperator{\ran}{ran}
\DeclareMathOperator{\Th}{Th}
\DeclareMathOperator{\tp}{tp}
\newcommand{\cmodels}{\Vdash}

\title[Cologic of closed covers of compacta]{Cologic of closed covers of compacta and the pseudo-arc}
\author{Kentar\^o Yamamoto}
\date{\today}
\begin{document}
\maketitle
\begin{abstract}
  A formal system called cologic is proposed for the study of compacta.
  A counterpart of countable model theory is developed for this system,
  and it is applied to model theory of the pseudo-arc.
\end{abstract}
\section{Introduction}
This article describes a way to import some techniques of first-order model theory,
especially countable model theory,
into the study of (metrizable) compacta.
The key feature of our approach that sets it apart from others
(e.g., \cite{Bankston1984, Flum1980}) is
that it aligns with a successful viewpoint in constructing and analyzing compacta:
compacta as the quotients of projective \Fr\ limits.
Pioneered by Irwin and Solecki~\cite{Irwin2006ProjectiveFL},
projective \Fr\ limits are
zero-dimensional compact spaces equipped with a closed binary relation,
called pre-spaces,
constructed as the inverse limits of
finite graphs and certain surjections.
Desired compacta are then obtained as the natural quotients of pre-spaces.
Kruckman~\cite{kruckman17:_first_order_logic_local_finit} proposes
that a formal system he calls \emph{cologic} be studied.
Cologic describes pre-spaces in terms of their finite quotients (graphs)
and surjections among them,
while ordinary first-order logic describes infinite structures
in terms of their finite substructures (tuples)
and injections among them.
Then familiar theorems of countable model theory,
e.g.,
that \Fr\ limits in a finite relational language have an $\omega$-categorical theory with quantifier elimination,
would obtain in regard to cologic and projective \Fr\ limits.

In this article, we explore the possibility of using cologic
to study compacta directly.
As we will see in \S~\ref{sec:background},
certain pre-spaces whose natural quotient is the given compactum~$K$
and bases of closed sets of $K$ consisting of regular closed sets
are in one-to-one correspondence.
Therefore, Kruckman's cologic might be called the logic of certain bases
of compacta,
while the cologic developed here would be the logic of compacta themselves.
This is the motivation behind developing the new approach.
Our cologic will describe compacta in terms of their
finite closed covers (of certain form) and how they refine each other.

Our direct approach poses a few obstacles.
For instance, there is no compactness theorem for cologic of compacta themselves.
One can still appeal to compactness theorem for cologic of pre-spaces
to obtain a model,
which, however, may have a theory different from that of its natural quotient.
Moreover,
the relationship between cologic and countable model theory,
including the \Fr\ theory,
is not obvious with cologic of compacta themselves.
This is because previous authors only considered---for good reason---%
second-countable pre-spaces,
and by the aforementioned one-to-one correspondence, Kruckman's cologic
concerns certain \emph{countable} bases of compacta.

In spite of this, we aim to demonstrate the utility of cologic of
compacta themselves
by proving a counterpart of Vaught's theorem on homogeneity
of countable atomic structures
and applying it to Irwin's and Solecki's original example
the pseudo-arc.
We define the notion of the type realized by each given finite cover
consisting of regular closed sets of a compacta.
Then principal types are defined in the usual manner.
We introduce the notion of cofinal atomicity (Definition~\ref{defi:cofat}),
which is weaker than the requirement that
every such finite cover realize a principal type.
Cofinal atomicity implies homogeneity (Theorem~\ref{thm:vaught}):
two finite covers of the correct kind of a compactum~$K$
are conjugate under an autohomeomorphism of $K$
if $K$ is cofinally atomic with the two covers realizing the same type.
Moreover, we show that if a compactum is the quotient of
two second-countable pre-spaces satisfying mild conditions,
the two pre-spaces are isomorphic
provided that they have the same cological theory and that they are cofinally atomic (Theorem~\ref{thm:uniqueness}).
This justifies  the use of pre-spaces in place of compacta themselves
in model theory of compacta in certain cases.
We then turn to the specific example of the pseudo-arc.
Notwithstanding the aforementioned general fact,
Irwin's and Solecki's projective \Fr\ limit has the same cological theory
as its natural quotient, the pseudo-arc (Theorem~\ref{thm:IR pre-space is elementary substr}).
We see this by adapting the notion of elementary substructures.
We use this fact to establish the cofinal atomicity of the pseudo-arc
(Corollary~\ref{cor:main}).
The conclusion is that we can view the homogeneity of the pseudo-arc
(Corollary~\ref{cor:homog})
in the sense above as a consequence of its logical property.

This article is organized as follows:
In \S~\ref{sec:background},
we explain the mild condition we impose on pre-spaces and
its relationship with our focus on finite covers consisting of
regular closed sets of compacta.
In \S~\ref{sec:covers},
we define the kind of covers that we study
and establish basic facts on covers and the refinement relation between them.
The formal system of cologic of both pre-spaces and compacta is introduced
in \S~\ref{sec:basics-cologic},
where we also develop the counterpart of countable model theory.
Finally,
\S~\ref{sec:appl-pseudo-arc} is
where we apply everything thus far
to the pseudo-arc
and prove the main result.

In our presentation, we prioritize establishing
the aforementioned facts on the pseudo-arc
without much regard to presenting the most generalized results
or transferring as many theorems from classical model theory as possible.
It is plausible that our results are applicable to other compacta similar
enough to the pseudo-arc
and that more counterparts of countable model theory is true of cologic
of compacta.
It remains an important future task to explore these possibilities.

\section{Background}\label{sec:background}
The approach taken in this article is that only pre-spaces
whose natural surjections are irreducible are \emph{bona fide} models of Kruckman's cologic
and that compacta are described in terms of the family of their covers
by regular closed subsets that merely touch each other.
In this section, we present some evidence suggesting that this is a reasonable
framework together with relevant basic definitions.

We find it useful to recall basic facts on regular closed sets here.
\newcommand{\regular}{\mathcal R}
We write $\regular(X)$ for the family of regular \emph{closed}
subsets of a topological space~$X$.
With a suitable set of operations,
$\regular(X)$ is a complete Boolean algebra:
$0$ is the empty set, $1$ is $X$,
$\vee$ is the set-theoretic union,
and, most importantly, $\neg$ is given by $\neg F = \overline{X \setminus F}$.
Then, as a derivative operation, $\wedge$ satisfies
$F \wedge F' = \overline{(F \cap F')^\circ}$.

\newcommand{\prox}{\mathrel{\delta}}
A useful additional structure on $\regular(X)$ is the \emph{proximity relation}%
~$\prox$ defined by
$a \prox b \iff a \cap b \neq \0$.
for $a, b \in \regular(X)$.
This makes $\regular(X)$ a \emph{contact algebra}~$(\regular(X), \delta)$.
Here, following \cite{Koppelberg2012}, a contact algebra is Boolean algebra~$B$ expanded with a symmetric binary
relation~$\prox$ satisfying the following analog of the axioms of proximity spaces:
\begin{align}
&  0 \centernot\prox a \notag,\\
&  a \prox a   \text{ if }  a\neq 0, \notag\\
&  x \prox (y \vee z)  \iff  x \prox y \text{ or } x \prox z.\label{eq:contact vee}
\end{align}

\begin{defi}
  \begin{enumerate}[(i)]
  \item A \emph{compactum} is a compact \emph{metrizable nonempty}
    space.  
  \item A \emph{pre-space} is a pair $X_0 = (X_0, E)$ of a totally
    disconnected compact space~$X_0$ and a equivalence
    relation~$E \subseteq X_0 \times X_0$.  
    Important examples of pre-spaces are finite graphs,
    which, in this paper, are regarded as discrete spaces
    and supposed to have reflexive edges at all vertices.
  \item We write $|X_0|$ for the \emph{natural quotient} of a pre-space~$X_0$,
    i.e., the set of $E$-classes of $X_0$ with the quotient topology.
  \item Let $X_0 = (X_0, E)$ be a pre-space and $X$ be a compactum.
    A continuous surjection $\pi: X_0 \twoheadrightarrow X$
    \emph{induces} $E$ if $x\mathbin{E}y \iff \pi(x) = \pi(y)$.
    For instance, this occurs when $X$ is the natural quotient of $X_0$
    and $\pi$ is the natural surjection.
  \item Let $X_i = (X_i, E_i)$ ($i <2$) be pre-spaces.
    A continuous surjection $f: X_0 \twoheadrightarrow X_1$
    is an epimorphism in the sense of Irwin and Solecki~\cite{Irwin2006ProjectiveFL},
    or an \emph{Irwin-Solecki epi},
    if
    \begin{align*}
      (x_1, y_1) \in E_1 &\implies (f(x_1), f(y_1)) \in E_2,\\
      (x_2, y_2) \in E_2 &\implies \aru*{x_1 \in f^\inv(x_2)}\aru*{y_1 \in f^\inv(x_2)}
      [(x_1, y_1) \in E_2].
    \end{align*}
  \item A surjection~$\pi: Y \to Z$ is \emph{irreducible} if
    there exists a proper closed subset $Y' \subseteq Y$ such that
    $\pi``Y' = Z$.
\end{enumerate}
\end{defi}
Given a compactum~$X$,
we may construct  construct a pre-space whose quotient is $X$
as was described by Woods~\cite{98ca69f0-6429-3dcd-bf45-5f38c098a482}.
It is convenient to define it via a general construction for
arbitrary contact algebras~$B$.
Given such $B$,
let $S(B)$ be the Stone space of the Boolean algebra $B$:
as a set, $S(B)$ is the set ultrafilters of $\mathcal B$,
and its topology is generated by sets of the form~$[F]= \{x \in S(B) \mid F \in x\}$ where $F \in \mathcal B$.
Define $E \subseteq S(B) \times S(B)$ by
$x \mathrel{E} y \iff \aru*{a \in x}\aru*{b \in y}[a \prox y]$.
This makes $S(B) := (S(B), E)$ a pre-space.
Now consider $\mathcal B_0 \subseteq \regular(X)$ a basis of closed sets.
Let $\mathcal B$ be the Boolean algebra generated 
(as a subalgebra of $\regular(X)$)
by $\mathcal B_0$.
Let $a_{\mathcal B_0}X = S(\mathcal B$).
Then
one can show that 
$\pi: a_{\mathcal B_0}X \to K$ defined by $x \mapsto \bigcap x$ is a continous surjection
and that
 $E = \{(x, y) \in a_{\mathcal B_0}X \times X \mid \pi(x) = \pi(y)\}$.
Then $a_{\mathcal B_0}X := (a_{\mathcal B_0}X, E)$ is a prespace and $|(a_{\mathcal B_0}X, E)| \cong X$.
Note that $\pi``[F] = F$ for every regular closed $F \subseteq X$.

It is easy to show that the continuous surjection
constructed in the preceding paragraph is irreducible \cite[6H.(3)]{Porter1988}.
This is not true of every natural surjection from a pre-space.
In fact, there are projective \Fr\ limits whose natural surjections
are not irreducible.
\begin{exa}
Let $G_n$ ($n < \omega$) be the graph $2^n \sqcup \{*\}$
with the only non-reflexive edge between $1^n$ and $*$,
where $2^n$ is the set of binary strings of length~$n$.
Let $f_{mn}: G_n \twoheadrightarrow G_m$ ($m \le n < \omega$)
be Irwin-Solecki epis defined by 
$f_{mn}(\sigma) = \sigma \upharpoonright n$
for any string~$\sigma \in 2^m$ and $f_{mn}(*) = *$.
Then $\varprojlim (G_n, f_{mn})$
is the pre-space $(2^\omega \sqcup \{\star\}, E)$,
where $E$ consists of the reflexive edges and another between $\star$
and the rightmost path.
Then $|(2^\omega \sqcup \{\star\}, E)| \cong 2^\omega$.
The canonical surjective continuous map $\pi :(2^\omega \sqcup \{\star\}, E) \twoheadrightarrow |(2^\omega \sqcup \{\star\}, E)|$
is not irreducible:
$\operatorname{ran} \pi = \pi`` (\operatorname{dom} \pi \setminus \{\star\})$;
since $\star$ is isolated,
$\operatorname{dom} \pi \setminus \{\star\}$ is a closed subset.
Even though we do not go into technical details,
but this example~$(2^\omega \sqcup \{\star\}, E)$ may be thought of as
the projective \Fr\ limit of the category of
graphs of the form $G_n$ and suitable morphisms.
\end{exa}
One must admit that the example above is neither natural nor economical.
On the other hand, focusing on pre-spaces those natural surjections are
irreducible has the benefit of clarifying the correspondence between
the finite quotients of pre-spaces by Irwin-Solecki epis,
which Kruckman's cologic describes,
and the finite covers of compacta by regular closed sets
that merely touch
which our cologic as defined later describes.
Here, a finite regular closed cover~$(F_i)_{i<n}$ \emph{merely touch} 
if $(F_i \cap F_j)^\circ = F_i^\circ \cap F_j^\circ = \0$ if $i \neq j$.

The correspondence described in the preceding paragraph should
more accurately be stated as follows.
First, note that every quotient of a pre-space~$X_0$ by an Irwin-Solecki epi
has an isomorphic copy as a partition of $X_0$ by clopen sets.
Then for a fixed compactum~$X$,
there is a special pre-space~$X_0$ and a natural surjection $\pi$
such that the finite quotients~$G = (\{P_1, \dots, P_{|G|}\}, E(G))$
of $X_0$ by Irwin-Solecki epis
and the finite covers of $X$ by regular closed sets that merely touch
are in one-to-one correspondence via $G \mapsto \{\pi`` P_1, \dots, \pi``P_{|G|}\}$,
where $P_1, \dots, P_{|G|}$ are clopen subsets of $X_0$.
Moreover, for a restriction of the correspondence onto a smaller family of
such covers of $X$,
the pre-space $X_0$ can be chosen to be ``smaller.''

To see this, let us first consider how to construct finite quotients of pre-spaces
from finite covers of a compactum~$X$ by regular closed sets that merely touch.
Take $\mathcal B_0$ be $\regular(X)$ or large enough
to contain all regular closed sets occurring in such covers that we want to deal with.
Let $X_0 = a_{\mathcal B_0} X$.
We have $\pi ``[F_i] = F_i$ as noted before.
To show that $([F_i])_i$ is a partition of $X_0$,
note that by assumption, $F_i \cap F_j$ has an empty interior,
i.e., $F_i \wedge F_j = 0$.
By general facts on the Stone duality,
$[F_i] \cap [F_j] = [F_i \wedge F_j] = [0] = \0$ as desired.

For the other direction of the correspondence,
suppose that $X_0 = (X_0, E)$ is a pre-space and that
the natural continuous surjection~$\pi:X_0 \to X$
is irreducible.
Let $(P_i)_{i<n}$ be a partition of $X_0$ by clopen subsets.
Then it may be shown that the image of $(P_i)_{i<n}$ under $\pi$
is a regular closed cover of $X$ that merely touch
\cite[Theorem~6.5(d)]{Porter1988}%
\footnote{The theorem is about $\theta$-continuous maps,
  example of which are continuous maps.}
as clopens in $X_0$ are regular closed.

It turns out that the aforementioned correspondence also preserves the contact
algebra structure.
To see this,
\newcommand{\clopen}{\mathcal B}
we write $\clopen(X_0)$ for the family of clopen subsets of a  a pre-space~$(X_0, E)$
and turn it into a contact algebra~$(\clopen(X_0), \delta)$
by declaring $a \prox b$ if and only if $x \mathbin{E} y$ for some $x\in a$
and $y \in b$.
We remark in passing that $S(\clopen(X_0)) \cong X_0$.
\begin{prop}
  Let $X$ be a compactum.
  \begin{enumerate}[(i)]
  \item\label{item:1}
    Let $X_0 = (X_0, E)$ be an arbitrary pre-space
    and $\pi: X_0 \twoheadrightarrow X$ 
    an irreducible continuous surjection inducing $E$.
    Then $\pi$
    induces the contact algebra embedding~$e_\pi: \clopen(X_0) \hookrightarrow \regular(X)$
    by $P \mapsto \pi``P$.
  \item \label{item:2}
    In \ref{item:1}, the image of $e_\pi$ is a basis of closed sets.
  \end{enumerate}
\end{prop}
\begin{proof}
  \begin{enumerate}[(i)]
  \item We have already quoted \cite[Theorem~6.5(d)]{Porter1988}
  which, in fact, claims that $e_\pi$ is a Boolean algebra embedding.
  It remains to see that $e_\pi$ preserves and reflects the proximity relation.
  This is done by observing:
  \begin{align*}
    a \prox b
    &\iff \aru*{x \in a} \aru*{y \in b} [x \mathrel{E} y] &\text{(definition of $\prox$ in $\clopen(X_0)$)}\\
    &\iff \aru*{x \in a} \aru*{y \in b} [\pi(x)= \pi(y)] &\text{($E$ is induced by $\pi$)}\\
    &\iff \aru{z} z \in \pi``a \cap \pi``b\\
    &\iff \pi`` a \prox \pi``b. &\text{(definition of $\prox$ in $\regular(X)$)}
  \end{align*}
\item
  It suffices to show that
  if $x \not \in F$,
  where $F \cup \{x\} \subseteq X$, and $F$ is closed,
  then
  there exists $a \in \clopen(X_0)$
  such that $F \subseteq \pi``a \not\ni x$.
  This follows from
  the disjointness of $\pi^\inv(F)$ and $\pi^\inv(x)$,
  the normality of $X_0$,
  Urysohn's Lemma,
  and the strong zero-dimensionality of $X_0$
  (see, e.g., \cite[\S~6.2]{engelking1989general}). \qedhere
\end{enumerate}
\end{proof}

In light of \ref{item:2} of the Proposition,
every irreducible continuous surjection~$\pi:(X_0, E) \twoheadrightarrow X$
inducing $E$ 
arises as a natural surjection of a pre-space.
Indeed,
there exists a isomorphism~$\phi: X_0 \to a_{\ran e_\pi}X$
such that $k \circ \phi = \pi$,
where $k: a_{\ran e_\pi}X \twoheadrightarrow X$ is the canonical surjection.

\emph{In what follows, whenever we consider a continuous surjection map~$\pi$ from a pre-space~$(X_0, E)$ whose equivalence relation is induced by $\pi$, we assume that $\pi$ is irreducible.}

\section{Covers}
\label{sec:covers}

In the foregoing section,
we have established finite covers consisting of regular closed sets
that merely touch each other as the central object of study
in the present article.
It has already been suggested there
that such covers
may be best handled at the level of contact algebras.
In this section, we devise definitions based on this idea.
We then move on to the issue of refinement of covers,
which is central in continuum theory.
We introduce the notion of patterns and other related definitions
and prove fundamental facts on covers and refinement thereof.

\subsection{Covers}
\newcommand{\covers}{\mathcal K}
\begin{defi}
  \begin{enumerate}[(i)]
  \item Let $B$ be a contact algebra.
    A \emph{finite minimal covers of $B$ that merely touch}
    is a
    finite set $C \subseteq B$ such that
    $\bigvee C = 1_B$, $a \neq b \implies a \wedge b = 0$,
    and $a \le b \implies a = b$
    for all $a, b \in B$
    (\emph{a fortiori}, members of finite minimal covers are nonzero).
    Let $\covers(B)$ be the family of
    such finite sets.
    We write $\covers(X)$ and $\covers(X_0)$
    for $\covers(\regular(X))$ and $\covers(\clopen(X_0))$
    when $X$ and $X_0 = (X_0, E)$ are a compactum and a pre-space, respectively.
  \item 
    We call compacta, pre-spaces, and contact algebras \emph{models}.
    They are exactly the mathematical objects that can be an argument of $\covers^*(-)$
    or on the left-hand side of $\cmodels$ (as defined later).
  \end{enumerate}
\end{defi}
Note that elements of $\covers(\regular(X))$ and the images of elements of $\covers(X_0)$ under $e_E$
are finite regular closed covers that merely touch
in the sense defined topologically in \S~\ref{sec:background}.
\begin{defi}
    Let $\mathfrak M$ be a model.
  \begin{enumerate}[(i)]
  \item A \emph{good tuple} is a non-repeating tuple enumerating an element of
    $\covers(\mathfrak M)$
    The set of such tuples is denoted as $\covers^*(M)$
    whereas $\covers^n(B) := \covers^*(B) \cap B^n$ for $n < \omega$.
  \item 
    A good tuple~$\overline a$ is a \emph{chain} if
    $a_i \wedge a_j \neq \0 \implies |i - j | \le 1$.
  \end{enumerate}
\end{defi}
\newcommand{\qftp}{N^1}
\newcommand{\chaingraph}{\acute}
We write $\qftp(\overline a)
\subseteq [\omega]^{\le 2}$ for the underlying graph
of the nerve
complex of a good tuple~$\overline a$.
We write $\chaingraph n$ for $\qftp(\overline a)$ for any chain of length~$n$.

\subsection{Arrangement following and refinement}
The notion of patterns is adapted from Oversteegen and Tymchatyn~\cite{oversteegen86}.

\begin{defi}
  Let $B$ be a contact algebra.
  \begin{enumerate}[(i)]
  \item Let $\overline a \in \covers^m(B)$, $\overline b \in \covers^n(B)$,
    $m \le n < \omega$.
    We say that $\overline b$ \emph{follows an arrangement}~$f: m \twoheadrightarrow n$
    in $\overline a$
    if for every $i< n$, $b_j \le a_{f(j)}$.
    If there is an arrangement that $\overline b$ follows in $\overline a$,
    then $\overline b$ \emph{refines} $\overline a$.
  \item 
    The surjection~$f$ is a \emph{pattern} if $|i - j| \le 1 \implies |f(i) - f(j)| \le 1$.
  \end{enumerate}
\end{defi}
The term \emph{pattern} is intended to be used
when a chain follows a pattern in another.

\begin{lem}\label{lem:refinement}
  Let $B$ be a contact algebra
  and
  $\overline a, \overline b \in \covers^*(B)$,
  and assume that $\overline b$ follows an arrangement~$f: m \twoheadrightarrow n$ in $\overline a$.
  \begin{enumerate}[(i)]
  \item\label{item:unique} If $b_j \wedge a_i \neq 0$, then $i = f(j)$,
    and $b_j \le a_i$.
  \item \label{item:consolidation}
    The cover~$\overline b$ is a consolidation of $\overline a$.
    That is, $a_i = \bigvee_{j \in f^\inv(i)} b_j$ for all $i < n$.
  \item Let $\overline a', \overline b' \in \covers^*(B)$,
    and suppose that $\overline b'$ follows $f$ in $\overline a'$.
    If $\qftp(\overline b) = \qftp(\overline b')$,
    then $\qftp(\overline a) = \qftp(\overline a')$.\label{item:qftp}
  \end{enumerate}
\end{lem}
\begin{proof}
  \begin{enumerate}[(i)]
  \item
    Since $b_j \le a_{f(j)}$, we have $a_{f(j)} \wedge a_i \neq 0$.
    Hence $f(j) = i$.
    Since $\overline b$ follows $f$ in $\overline a$,
    $b_j \le a_{i'}$ for some $i'$;
    a fortiori, $b_j \wedge a_{i'} \neq 0$;
    by repeating the same argument, $i' = f(i) = i$.
  \item
    Let $\widetilde a_i$ be the right-hand side.
    That $a_i \ge \widetilde a_i$ is obvious.
    To see $a_i \le \widetilde a_i$,
    suppose not,
    and let $r = a_i \wedge \neg \widetilde a_i \neq 0$.
    Since $\overline b$ is a cover,
    there exists $j < m$ such that $b_j \wedge r \neq 0$, whence $b_j \wedge a_i \neq \0$.
    Since $\overline b$ refines $\overline a$,
    by item~\ref{item:unique},
    $j \in f^\inv(i)$, and $b_j \le a_i$.
    Hence $b_j \wedge r \le
    b_j \wedge \neg \widetilde a_i =
    b_j \wedge \bigwedge_{j' \in f^\inv(i)} \neg b_{j'} = 0$,
     a contradiction.
   \item
     This follows immediately from the preceding item
     and the axiom~(\ref{eq:contact vee}) of contact algebras.\qedhere
  \end{enumerate}
\end{proof}
Item \ref{item:unique} of the lemma
allows us to speak of \emph{the} arrangement that a
cover follows in another.
We will make use of this without specific reference to this Lemma.

Let $\overline b \in \covers^n(B)$,
and $f: m \twoheadrightarrow n$ is a surjection.
By the preceding Lemma, there exists a unique element $\covers^m(B)$
following the arrangement~$f$ in $\overline b$,
and it is a consolidation of $\overline b$.
We write $f^* \overline b$ for this.
It is easy to see that
\begin{equation}
  \label{eq:consolidation}
  (f^* \overline b)_i = \bigvee_{j \in f^\inv(i)} b_j
\end{equation}
for each $i < n$.

\begin{lem}\label{lem:undo}
  Let $\overline a \in \covers^n(B)$,
  $\overline b \in \covers^m(B)$,
  and $f: m \twoheadrightarrow n$ and $g: l \twoheadrightarrow m$
  be surjections.
  Suppose that $\overline b$ follows the arrangement~$f \circ g$ in $\overline a$.
  Then the consolidation~$g^* \overline b$ follows $f$ in $\overline a$.
\end{lem}
\begin{proof}
  By calculation using (\ref{eq:consolidation}) and
  Lemma~\ref{lem:refinement}.\ref{item:consolidation}.
\end{proof}

We conclude the section by stating
the following elementary fact on chains without a proof.
\begin{lem}\label{lem:obvious}
  Let $B$ be a contact algebra.
  \begin{enumerate}[(i)]
  \item \label{item:automatically a pattern}
    If $\overline a, \overline b \in \covers^*(B)$ are chains,
    and $\overline b$ follows $f: |\overline b| \twoheadrightarrow |\overline a|$ in $\overline a$,
    then $f$ is a pattern.
  \item \label{item:pullback of a chain by a pattern is a chain}
    If $\overline a \in \covers^*(B)$ is a chain,
    and $f: |\overline a| \twoheadrightarrow n$ is a pattern,
    then $f^*\overline a$ is a chain.
  \end{enumerate}
\end{lem}

\section{Cologic}\label{sec:basics-cologic}
\newcommand{\trivial}{1}
\subsection{Basics}
We introduce the formal system of cologic.
To make the definitions uniform between Kruckman's cologic of pre-spaces
and our cologic of compacta, the formalism of contact algebras is exploited.

Formally, variables in cologic are natural numbers,
and \emph{contexts} are finite initial segments of $\omega$.
In practice, we will be lax about what can be a context and allow any finite set
with a suitable bijection onto a natural number implicit.
For instance, we say that $x$ is a variable
and that $\{x, y\}$ is a context, etc.
We will sometimes extend this convention to the index sets of tuples
as demonstrated in the proof below.
\begin{lem}\label{lem:directed}
  Let $X$ be a compactum, and $\overline a, \overline b \in \covers^*(X)$.
  Then there exists a common refinement of $\overline a$ and $\overline b$
  in $\covers^*(X)$.
\end{lem}
\begin{proof}
  Let $I = \{(i, j) \in |\overline a| \times |\overline b| \mid a_i \wedge b_j \neq 0\}$.
  Define a tuple~$\overline c$ indexed by $I$ by
  $c_{(i, j)} = a_i \wedge b_j$.
  By the distribitive law of Boolean algebras, $\overline c$ is a cover.
  By definition, $\overline c$ refines both $\overline a$ and $\overline b$.
  It is clear that $\overline c$ merely touches.
\end{proof}
For each context~$n$, \emph{formulas} (of cologic) in context~$n$,
or \emph{$n$-formulas},
are generated by the following grammar:
\[
  \phi^n ::= \bot \mid G \mid \neg \phi^n \mid \phi^n \vee \phi^n \mid \caru{f}\phi^m
\]
where $G$ is a finite graph whose vertex set is $X$,
and $f: m \twoheadrightarrow n$ is  a surjection.
As usual, $\top$, $\land$, $\to$, and $\czenbu{f}$ are introduced as shorthand
(so, e.g., $\czenbu{f}\phi := \neg \caru{f} \neg \phi$).
A formula in the singleton context~$1 = \{0\}$ is a \emph{sentence},
and, as usual, a \emph{theory} in cologic is a set of sentences of cologic.

Let $B$ be an arbitrary contact algebra.
We define the satisfaction relation $B, \overline a \cmodels \phi$ of cologic
for good tuples $\overline a \in \covers^n(B)$ with the index set~$n$,
and $n$-formulas $\phi$ for each context~$n$.
This is done by recursion:
\begin{itemize}
\item
  $M, \overline a \cmodels \caru{f}\phi_0$
  if and only if
  there exists a good tuple $\overline b$ in $M$
  following the arrangement~$f$ in  $\overline a$
  such that $M, \overline b \cmodels \phi_0$;
\item
  $M, \overline a \cmodels G$
  if and only if
  $G = \qftp(\overline a)$; and
\item
  we have the obvious conditions corresponding to the Boolean connectives. 
\end{itemize}
For a sentence $\phi$, we define $B \cmodels \phi$ if and only if
$B, \trivial^B \cmodels \phi$, where $\trivial^B$ is the singleton tuple~$1$
which is good. 
The (full) cological theory $\Th_\cmodels(B)$ is
the set of sentences of cologic $\{\phi \mid B \cmodels \phi\}$.
For a compactum $X$,
by convention $X, \overline a \cmodels \phi \iff \regular(X), \overline a \cmodels \phi$.
For a pre-space $X_0 = (X_0, E)$,
by convention $X_0, \overline a \cmodels \phi \iff \clopen(X_0), \overline a \cmodels \phi
\iff e_E(\clopen(X_0)), e_E(\overline a) \cmodels \phi$.

\begin{thm}[Compactness]
  Let $T$ be a theory of cologic.
  If every finite subset $T'$ of $T$ is \emph{satisfiable},
  i.e.,
  there exists a \emph{pre-space}~$X_0$ with $X_0 \cmodels T$,
  then so is $T$.
\end{thm}
\begin{proof}[Proof sketch]
  This can be proved in two ways.
  Kruckman~\cite{kruckman17:_first_order_logic_local_finit} has
  a proof system complete for his cologic, from which
  the compactness theorem follows in the usual manner.
  Alternatively,
  we may use the compactness theorem of first-order logic:
  it is clear that our cologic can be translated into
  first-order logic in the language of contact algebras,
  whose axioms are elementary,
  and a pre-space satisfying the given theory of cologic
  is obtained via the duality, given by $S(-)$ and $\clopen(-)$,
  of pre-spaces and contact algebras
  as described in \S~\ref{sec:background}.
\end{proof}
The foregoing compactness theorem gives us pre-spaces
and not compacta \emph{per se}.
Nor is there a guarantee that those pre-spaces have canonical surjections
that are irreducible.
We have to bite the bullet for now and live with these facts.
On the other hand,
as we will see later,
a categoricity result of some sort (Corollary~\ref{cor:uniqueness})
states that, for some compactum~$X$,
pre-spaces~$(X_0, E)$ satisfying some mild condition
with (irreducible) continuous surjections onto $X$
inducing $E$ are actually unique
and suggests that there is a strong connection between
cologic of compacta and cologic of pre-spaces, the latter of which \emph{is} compact.

The following notion, which is the counterpart of elementary substructures,
will be of importance in \S~\ref{sec:appl-pseudo-arc}.
\begin{defi}
  For contact algebras $A \subseteq B$,
  $A$ is a \emph{elementary substructure} of $B$%
  \footnote{
    Perhaps it is better to say that $\covers(A)$ is an elementary substructure
    of $\covers(B)$ or that $\covers^*(A)$ is an elementary substructure of $\covers^*(B)$.}
  if 
  $A, \overline a \cmodels \phi \iff B, \overline a \cmodels \phi$
  for $n < \omega$,
  $\overline a \in \covers^n(A)$, and an $n$-formula~$\phi$.
\end{defi}

\begin{defi}\label{defi:back-and-forth}
  Let $A$ and $B$ be contact algebras.
  A \emph{back-and-forth system}~$I$ from $A$ to $B$
  is a set of pairs of tuples satisfying the following:
  \begin{enumerate}[(i)]
  \item if $(\overline a, \overline a')$,
    then $|\overline a| = |\overline a'| =: n$,
    $(\overline a, \overline a') \in \covers^n(A) \times \covers^n(B)$,
    and $\qftp(\overline a) = \qftp(\overline a')$;
  \item if $(\overline a, \overline a')$,
    and $\overline b \in \covers^*(A)$ follows $f$ in $\overline a$,
    then there exists
    $\overline b' \in \covers^*(B)$ following $f$ in $\overline a'$
    such that $(\overline b, \overline b') \in I$; and
    \label{item:forth}
  \item if $(\overline a, \overline a')$,
    and $\overline b' \in \covers^*(B)$ follows $f$ in $\overline a'$,
    then there exists
    $\overline b \in \covers^*(A)$ following $f$ in $\overline a$
    such that $(\overline b, \overline b') \in I$.
    \label{item:back}
  \end{enumerate}
\end{defi}

\begin{prop}\label{prop:back and forth}
  Let $A, B$ be contact algebras.
  If $I$ is a back-and-forth system from $A$ to $B$,
  and $(\overline a, \overline a') \in I$,
  then $A, \overline a \cmodels \phi \iff B, \overline a' \cmodels \phi$
  for all cological formula~$\phi$ in context~$|\overline a| = |\overline a'|$.
  In particular,
  if $A \subseteq B$,
  and
  $I$ contains (the graph of) the embedding~%
  $\covers^n(A) \hookrightarrow \covers^n(B)$ for each $n$,
  then $A$ is an elementary substructure of $B$.
\end{prop}
\begin{proof}
  This is proved in the same way as the usual analysis of Ehrenfeucht-\Fr{} games.
\end{proof}
In the case of first-order logic,
the requirement that the Ehrenfeucht-\Fr{} game has a uniform winning strategy for playing forever is too strong for elementary equivalence.
Likewise, the sufficient condition in Proposition~\ref{prop:back and forth}
is too strong.
Nevertheless, Proposition~\ref{prop:back and forth} is good enough for
our application in \S~\ref{sec:appl-pseudo-arc}.

\subsection{Analog of countable model theory}
We develop the counterpart of countable model theory,
where types over the empty set play an important role.
We do not define or use the notion of cological types over a nonempty
set of parameters,
so a cological type \emph{simpliciter}
is what would otherwise be called a cological type over $\emptyset$.
Some counterparts of the ordinary concepts on types in first-order logic
would pose issues in the absence of compactness theorem,
but note that the following definitions still make sense.
\begin{defi}
  Let $\mathfrak M$ be a model.
  \begin{enumerate}[(i)]
  \item Let $\overline a \in \covers^n(\mathfrak M)$ be a good tuple.
    The set of $n$-formulas of cologic
    \[
      \tp_\cmodels := \{\phi \mid \mathfrak M, \overline a \cmodels \phi\}
    \]
    is the \emph{type realized} by $\overline a$ in $\mathfrak M$.
    Such a type is an \emph{$n$-type}.
  \item Let $p$ be an $n$-type realized by some good tuple in $\mathfrak M$.
    The type~$p$ is \emph{generated} by  an $n$-formula~$\phi$ if
    for every $\psi \in p$
    and every $\overline a \in \covers^n(\mathfrak M)$,
    we have $M, \overline a \cmodels \phi \to \psi$
    (the last part is equivalent to $M \cmodels \czenbu{n \times 1}(\phi \to \psi)$,
    where $n \times 1$ is the unique surjection~$n \twoheadrightarrow 1$).
    A type generated by some formula is \emph{principal}.
  \end{enumerate}
\end{defi}


\begin{defi}\label{defi:cofat}
  A model~$\mathfrak M$ is \emph{cofinally atomic}
  if there exists a set~$F$ of good tuples in $\covers^*(\mathfrak M)$
  with the following properties:
  \begin{enumerate}[(i)]
  \item \label{item:closed eqv}
    $F$ is closed under cological equivalence of good tuples;
  \item \label{item:directed}
    the refinement relation on $F$ is directed;
  \item \label{item:cofinal}
    $F$ is cofinal in $\covers^*(\mathfrak M)$; and
  \item  \label{item:ppl}
    every $\overline a \in \mathcal F$ realizes a principal cological type.
  \end{enumerate}
\end{defi}

The following is the promised analog of Vaught's theorem on
homogeneity of countable atomic structures.
\begin{thm}\label{thm:vaught}
  Let $X$ be an infinite compactum.
  Suppose that it is cofinally atomic.
  Then for $\overline a^i \in \covers^*(X)$ ($i<2$),
  if they realize the same cological type,
  then
  there exists an autohomeomorphism $\sigma: X \to X$
  that maps $\overline a^1$ to $\overline a^0$.
\end{thm}

\begin{proof}
  Take a set~$F$ of good tuples witnessing the cofinal atomicity of $X$.

  We build, for each $i<2$,
  a sequence~$(\overline c^{i, j})_{j<\omega}$ of good tuples in $F$,
  such that $\overline c^{0, j}$ and $\overline c^{1, j}$
  realize the same cological type.
  We will also build an increasing sequence~$(s_j)_{j<\omega}$ of
  finite functions
  whose union will be a contact algebra automorphism~$s: B \to B$,
  where $B$ is a subalgebra of $\regular(X)$ and is a basis of
  closed sets.
  Of course, $|\dom s | = |\ran s| = \aleph_0$.
  We also implicitly build an enumeration~$(b^j)_{j<\omega}$ of $B$
  such that $b^j \in \langle \dom s_{j-1} \cup \ran s_{j-1}\rangle \cup B_0$,
  where $B_0$ is a fixed countable basis of closed sets of $X$;
  this can be done by extending the enumeration by finitely many elements
  every time we define $s_j$.

  We first construct $\overline c^{0,0}$ and $\overline c^{1, 0}$.
  Use the cofinality of $F$ to find $\overline c^{0,0}$
  following the arrangement~$f_0$ in $\overline a^0$.
  By Definition~\ref{defi:cofat}.\ref{item:ppl},
  the cological type of $\overline c^{0,0}$ is generated
  by, say, $\phi_0$.
  Since $X, \overline a^0 \cmodels \caru{f_0} \phi_0$,
  we have $X, \overline a^1 \cmodels \caru{f_0} \phi_0$ by hypothesis,
  whence there exists $c^{1,0}$ following the arrangement~$f_0$ in $\overline a^1$
  such that $X, c^{1, 0} \cmodels \phi_0$.
  Since $\phi_0$ generates the cological type of $\overline c^{0,0}$,
  the good tuple~$c^{1, 0}$ realizes it,
  whence $c^{1, 0} \in F$
  by Definition~\ref{defi:cofat}.\ref{item:closed eqv}.
  Finally, define $s^0 := \0$.

  Let $j' > 0$,
  and suppose that we have constructed $\overline c^{i,0}, \dots, \overline c^{i, j'-1}$
  ($i<2$)
  and $s^0,\dots, s^{j'-1}$.
  We construct $\overline c^{i, j'}$ ($i < 2$) and $s^{j'}$ as follows.
  Let $j = \lfloor j'/2 \rfloor$ and $i = j' - 2j$.
  Consider the good tuple $(b^j, \neg b^j)$,
  where $\neg$ is the Boolean complementation in $\regular(X)$.
  One may take a common refinement~$\overline a^{i,j}$
  of that tuple and $\overline a^{i, j-1}$
  by Definition~\ref{defi:cofat}.\ref{item:directed} and
  \ref{item:cofinal};
  let $f_j, g_j$ be the arrangement followed by $\overline a^{i, j}$ in $\overline a^{i, j-1}$ and in $(b^j, \neg b^j)$,
  respectively.
  The cological type of $\overline a^{i,j}$ is principal
  by Definition~\ref{defi:cofat}.\ref{item:ppl}, so that
  it is generated by, say, $\phi_j$.
  Now, $X, \overline a^{i,j-1} \cmodels \caru{f_j}\phi_j$.
  Since $\overline a^{i,j-1}$ realizes the same type as
  $\overline a^{1-i, j-1}$ by induction,
  it follows that  $X, \overline a^{1-i,j-1} \cmodels \caru{f_j}\phi_j$.
  Take $a^{1-i, j}$ witnessing the existential formula on the right-hand side,
  which, as before, realizes the same type as $\overline a^{i, j}$
  and thus in $F$.
  Let $b'^{j}$ be the first component of $g_j^* a^{1-i, j}$.
  We define
  \[
    s^{j'} = s^{j'-1} \cup
    \begin{cases}
      (b^j, b'^j), & (i = 0)\\
      (b'^j, b^j). & (i = 1)
    \end{cases}
  \]

  It is easy to see that the bijection $s: B \to B$
  is a contact algebra isomorphism.
  By construction, the image of $\overline a^0$ under $s$ is $\overline a^1$.
  By the duality between contact algebras and compacta,
  $s$ induces an autohomeomorphism $\sigma: X \to X$
  determined by $\sigma(x) = \bigcap s^{\inv}(\{b \in B \mid x \in b\})$ for $x \in X$.
\end{proof}

A similar argument shows the following:
\begin{thm}\label{thm:uniqueness}
  Let $X$ be an infinite compactum
  and $(X_i, E_i)$ be second-countable pre-spaces $(i < 2)$.
  Suppose that there exists a continuous surjection~%
  $X_i \twoheadrightarrow X$ that induces $E_i$ for $i < 2$.
  If $(X_i, E_i)$ ($i<2$) have the same cological theory
  and are both cofinally atomic,
  then
  there exists an isomorphism~$(X_0, E_0) \to (X_1, E_1)$
  that induces an autohomeomorphism on $X$.
\end{thm}
\begin{cor}
  Let $X$ be a cofinally atomic infinite compactum
  and $B_i$ be a countable contact algebra that is a basis of closed sets.
  If $B_i$ ($i<2$) are both elementary substructures of $\regular(X)$,
  then there exists an autohomeomorphism inducing an isomorphism $B_0 \to B_1$.
\end{cor}
This result may be summarized as a cofinally atomic compactum
having a canonical pre-space with the same cological theory.

\section{Application: The pseudo-arc}\label{sec:appl-pseudo-arc}
In this final section,
we apply the machinery we have developed to begin model theory of the pseudo-arc.
Most of the technical arguments are aimed toward Lemma~\ref{lem:back and forth}.
From that, it follows that that Irwin and Solecki's pre-space has
the same cological theory as the pseudo-arc (Theorem~\ref{thm:IR pre-space is elementary substr}).
The Lemma also implies the cofinally atomicity of the pseudo-arc
(Corollary~\ref{cor:main})
as well as its homogeneity (Corollary~\ref{cor:homog})
and the canonicity, in the sense commented at the end of the last section,
of Irwin and Solecki's pre-space
among others whose quotient is the pseudo-arc (Corollary~\ref{cor:uniqueness}).
An interesting feature of the arguments below is that we do not quote
well-known facts on the pseudo-arc;
we instead extract information on the continuum
only from Irwin and Solecki's \Fr\ class.

Since the projective \Fr\ theory itself is out of the scope of this article,
we recall only necessary definitions and facts from Irwin and Solecki's work~\cite{Irwin2006ProjectiveFL} in the following.
For more details, the reader is referred to their seminal paper or Kubi\'s~\cite{KUBIS20141755}.

Hereafter,
$\mathbf{\Psi}$ is Irwin and Sokecki's~\cite{Irwin2006ProjectiveFL} \Fr\ class:
\[
  \mathbf{\Psi} = \{G \mid \text{$G$ is a finite linear graph}\}.
\]
Being a \Fr\ class, $\mathbf\Psi$
has the \emph{joint projection property}:
\begin{quote}
  for every $L_0, L_1 \in \mathbf{\Psi}$,
  there exist $L \in \mathbf{\Psi}$ and Irwin-Solecki epis~$f_i : L \twoheadrightarrow L_i$;
\end{quote}
it also has the \emph{amalgamation property}:
\begin{quote}
  for every Irwin-Solecki epis $f_i: L_i \twoheadrightarrow L^-$,
  where $L_0, L_1, L^- \in \mathbf \Psi$,
  there exist $L^+ \in \Psi$ and an Irwin-Solecki epi $\tilde f_i: L^+ \twoheadrightarrow L_i$
  such that $f_0 \circ \tilde f_1 = f_1 \circ \tilde f_0$.
\end{quote}
Let
$\Psi_0 = (\Psi_0, E)$ be its \Fr\ limit, whose
quotient is the pseudo-arc $\Psi$.
The pre-space~$\Psi_0$ is the inverse limit of
$(L_n)_{n<\omega}$ and $\pi^m_n: L_m \to L_n$
where $L_n \in \mathbf{\Psi}$,
and $\pi^m_n$ is an epimorphism in their sense,
or an Irwin-Solecki epi,
such that for every Irwin-Solecki epi $f: L \to L_n$ for $L \in \mathbf{\Psi}$
there exists $g: L_m \to L$ such that $\pi^m_n = f \circ g$
(this is the condition (b) in the proof of \cite[Theorem~2.4]{Irwin2006ProjectiveFL}).
Following Kubi\'s~\cite{KUBIS20141755}, we call this the \emph{amalgamation property}
of $(L_n)$, a \emph{\Fr\ sequence} of $\mathbf{\Psi}$.
The pre-space~$\Psi_0$ is \emph{projectively ultrahomogeneous}:
for any $L \in \mathbf\Psi$ and Irwin-Solecki epis~$f_1, f_2: \Psi_0 \twoheadrightarrow L$,
there exists an isomorphism $\psi: \Psi_0 \to \Psi_0$ with $f_1 \circ \psi = f_2$.
It follows \cite[Lemma 2.3]{Irwin2006ProjectiveFL}
that $\Psi_0$ is \emph{weakly homogeneous} in the sense inspired by Hodges~\cite{Hodges1993}:
if $\phi: L^+ \twoheadrightarrow L^-$ and $\psi: \Psi_0 \to L^-$
are Irwin-Solecki epis, and $L^+, L^- \in \Psi$, then
there exists an Irwin-Solecki epi $\chi: \Psi_0 \twoheadrightarrow L^+$
with $\phi \circ \chi = \psi$.

\begin{lem}\label{lem:semilattice}
  Let $C$ be the subset of $\clopen(\Psi_0)$ consisting of
  elements occurring in some chain in $\covers^*(\Psi_0)$.
  The set~$C$ generates $\clopen(\Psi_0)$ as a $\vee$-semilattice.
\end{lem}
\begin{proof}
  Since $\clopen(\Psi_0)$ is generated by $C$ as a Boolean algebra,
  it suffices to show that for $a', b' \in C$,
  the meet~$a' \wedge b'$ is the join of some elements of $C$.
  Take chains~$\overline a, \overline b \in \covers^*(\Psi_0)$
  and indices~$i < |\overline a|, j < |\overline b|$
  such that $a_i = a'$ and $b_j = b'$.
  By the projection, there exists a chain~$c \in \covers^*(\Psi_0)$
  refining, and thus by Lemma~\ref{lem:refinement}.\ref{item:consolidation} consolidating,
  both $\overline a$ and $\overline b$.
  Let $f, g$ the pattern followed by $\overline a$ and $\overline b$,
  respectively, in $\overline c$.
  Each entry of $\overline c$ is in $C$,
  and we have $a' \wedge b' = \bigvee \{ c_k \mid k < |\overline c|, f(k) = i, g(k) = j\}$
  as desired.
\end{proof}

\begin{lem}\label{lem:density}
  Suppose that $\pi: \Psi_0 \twoheadrightarrow \Psi$
  is a continuous surjection inducing $E$.
  Let $A := e_\pi(\clopen(\Psi_0))$ and $B := \regular(\Psi)$.
  Every good tuple in $\covers^*(B)$ is refined by another in $\covers^*(A)$,
  which can be chosen to be a chain.
\end{lem}
\begin{proof}
  Let $m<\omega$ and $\overline a \in \covers^m(A)$ be arbitrary.
we claim that there exists a chain
  $\overline c \in \covers^n(A)$
  following an arrangement~$g: n \twoheadrightarrow m$
  in $\overline a$.
  To see this,
  recall that for each $i < m$,
  there exist $k_i < \omega$
  and chains~$\overline c^{ij} \in \covers^*(A)$ ($j < k_i$)
  such that $a_i = \bigvee_{j < k_i}c^{ij}_0$ by Lemma~\ref{lem:semilattice}.
  Then we may take $\overline c$
  to be a common chain refinement of the finitely many chains~$\overline c^{ij}$,
  which exists by the joint projection property
  of $\mathbf \Psi$.
\end{proof}

\begin{lem}\label{lem:back and forth}
  Suppose that $\pi: \Psi_0 \twoheadrightarrow \Psi$
  is a continuous surjection inducing $E$,
  and let $A := e_\pi(\clopen(\Psi_0))$ and
  $B := \regular(\Psi)$.
  Consider the binary relation\[I := \{(\overline a, \overline a') \mid \overline a \in \covers^*(A), \overline a' \in \covers^*(B), \qftp(\overline a) = \qftp(\overline a)\}.\]
  \begin{enumerate}[(i)]
  \item The relation~$I$ satisfies Definition~\ref{defi:back-and-forth}.\ref{item:back}.
  \item If $\overline a' \in \covers^m(B)$ is a chain,
    and $f: M \twoheadrightarrow m$ is a pattern,
    then there exists a chain~$\overline b' \in \covers^M(B)$
    following $f$ in $\overline a'$.
  \item The relation~$I$ satisfies Definition~\ref{defi:back-and-forth}.\ref{item:forth}.
  \end{enumerate}
\end{lem}
Note that Lemma implies that $I$ is a back-and-forth system from $A$ to $B$.
\begin{proof}
  \begin{enumerate*}[(i), mode=unboxed]
  \item\label{item:back proof}
    Let $\overline a \in \covers^m(A)$,
    $\overline a' \in \covers^M(B)$,
    $\overline b' \in \covers^M(B)$,
    and $f: M \twoheadrightarrow m$.
    Assume that $\overline b'$ follows the arrangement~$f$ in $\overline a'$.
    We are to find $\overline b \in \covers^M(A)$
    following the arrangement~$f$ in $\overline a$.

    First, there exists a chain
    $\overline c \in \covers^n(A)$
    following an arrangement~$g: n \twoheadrightarrow m$
    in $\overline a$ by Lemma~\ref{lem:density}.

    Next, by Lemma~\ref{lem:directed}, there exists a
    $\overline c' \in \covers^N(B)$ refining both $\overline b'$ and $\overline c$;
    let $\tilde g$, $\tilde f$ be the arrangements that $\overline c'$ follows
    in $\overline b'$ and $\overline c$, respectively.
    Note that the arrangement that $\overline c'$ follows in $\overline a$
    is
    \begin{equation}
      \label{eq:arr}
      g \circ \tilde f = f \circ \tilde g.
    \end{equation}
    Again by Lemma~\ref{lem:density},
    we may assume without loss of generality that $\overline c'$
    is a chain.
    By Lemma~\ref{lem:obvious}.\ref{item:automatically a pattern}, $\tilde f$ is a pattern.
    
    It follows from the weak homogeneity of $\Psi_0$ that
    there exists $\overline c'' \in \covers^N(A)$ following the pattern~$r$
    in $\overline c$.
    By (\ref{eq:arr}),
    the chain $\overline c''$ follows the arrangement~$g \circ \tilde f
    = f \circ \tilde g$
    in $\overline a$.
    Define $\overline b := \tilde g^* \overline c''$.
    By Lemma~\ref{lem:undo},
    $\overline b$ follows the arrangement~$f$ in $\overline a$.
    To conclude that $(\overline b, \overline b') \in I$, note that
    $\qftp(\overline b) = \qftp(\overline b')$,
    and thus $(\overline b, \overline b') \in I$,
    by Lemma~\ref{lem:refinement}.\ref{item:qftp}.
    \\
  \item By Lemma~\ref{lem:density},
    there exists a chain~$\overline c \in \covers^*(A)$ refining $\overline a'$;
    let $g$ be the arrangement that $\overline c$ follows in $\overline a'$.
    Let $M' = |\overline c|$.
    Consider the finite graphs~$\chaingraph M, \chaingraph m, \chaingraph{M'} \in \mathbf \Psi$.
    Since $f: \chaingraph M \twoheadrightarrow \chaingraph m$
    and $g: \chaingraph{M'} \twoheadrightarrow \chaingraph m$
    are Irwin-Solecki epis,
    by the amalgamation property of $\Psi$
    there exist $G \in \mathbf \Psi$ and Irwin-Solecki epis
    $\tilde g: G \twoheadrightarrow \chaingraph M$
    and $\tilde f: G \twoheadrightarrow \chaingraph{M'}$
    such that $g \circ \tilde f = f \circ \tilde g$.
    Without loss of generality, we may assume that the domain of $G$
    is an initial segment of $\omega$,
    in which case $\tilde f$ and $\tilde g$ are patterns.
    By the weak homogeneity, we may take $\overline c' \in \covers^*(A)$
    following the pattern~$\tilde f$;
    $\overline c'$ follows the arrangement~$g \circ \tilde f$.
    Let
    $\overline b' = \tilde g^* \overline c'$.
    By Lemma~\ref{lem:undo}, $\overline b'$ follows $f$ in $\overline a'$,
    and by Lemma~\ref{lem:obvious}.\ref{item:pullback of a chain by a pattern is a chain}, $\overline b'$ is a chain.

    \\
  \item Let $(\overline a, \overline a') \in I$,
    and assume that
    $\overline b \in \covers^*(A)$ follows $f$ in $\overline a$.
    By Lemma~\ref{lem:density},
    we may find a chain~$\overline c' \in \covers^*(B)$ following, say, $g$
    in $\overline a'$.
    By \ref{item:back proof},
    take $\overline c \in \covers^*(A)$ following $g$ in $\overline a$.
    Take a common refinement~$\overline d \in \covers^*(A)$ of $\overline c$
    and $\overline c'$ by Lemma~\ref{lem:directed};
    by Lemma~\ref{lem:density} we may assume that it is a chain.
    Let $\tilde g, \tilde f$ be the arrangements that $\overline d$
    follows in $\overline c$ and $\overline c'$, respectively.
    By computing the arrangement that $\overline d$ follows in $\overline a$
    in two ways, $f \circ \tilde g = g \circ \tilde f$.
    Let $\overline b' = \tilde g^* \overline d$.
    By Lemma~\ref{lem:undo},
    $\overline b'$ follows $f$ in $\overline a'$,
    and by Lemma~\ref{lem:refinement}.\ref{item:qftp},
    $\qftp(\overline b) = \qftp(\overline b')$
    (whence $(\overline b, \overline b') \in I$).
  \end{enumerate*}
\end{proof}

\begin{thm}\label{thm:IR pre-space is elementary substr}
  Suppose that $\pi: \Psi_0 \twoheadrightarrow \Psi$
  is a continuous surjection inducing $E$.
  Then $A := e_\pi(\clopen(\Psi_0))$ is an elementary substructure of
  $B := \regular(\Psi)$.
\end{thm}
\begin{proof}
  Use the preceding Lemma and Theorem~\ref{prop:back and forth}.
\end{proof}

The preceding Theorem can alternatively be proved from the following
consequence of Lemma~\ref{lem:back and forth}.

\begin{cor}
  The cological theory of $\Psi$ eliminates quantifiers:
  for each cological formula~$\phi$ in context~$n$, there exists a
  $\langle \bullet \rangle$-free formula~$\phi^*$ in context~$n$
  such that \[\Psi \cmodels [n \times 1](\phi^* \leftrightarrow \phi),\]
  where $n \times 1$ is the unique arrangement~$n \twoheadrightarrow 1$.
\end{cor}

\begin{cor}\label{cor:principal}
  Let $n < \omega$ be arbitrary.
  All chains in $\covers^n(\Psi)$ realize the same cological type,
  which is principal.
\end{cor}
\begin{proof}
  The first claim follows from Lemma~\ref{lem:back and forth}.
  To see the second claim,
  note that for every formula~$\phi \in p$,
  \[
    \Psi \cmodels \czenbu{n \times 1}
    (\chaingraph n \to \phi)
  \]
  (recall that $\chaingraph n = \qftp(\overline a)$).
  In other words, $p$ is generated by $\chaingraph n$.
\end{proof}
The same argument shows that every $\overline a \in \covers^*(\Psi)$
realizes a principal type, generated by $\qftp(\overline a)$.
The focus here is instead to see that the homogeneity results at the end
of this article
follows from the weaker condition of cofinal atomicity.
\begin{cor}\label{cor:main}
    The pseudo-arc~$\Phi$ is cofinally atomic.
\end{cor}
\begin{proof}
  This follows form Corollary~\ref{cor:principal}
  and the cofinality of chains.
\end{proof}
\begin{cor}\label{cor:homog}
  For $\overline a^i \in \covers^*(X)$ ($i<2$),
  if they are of the same nerve (e.g., if they are both chains),
  then
  there exists an autohomeomorphism $\sigma: X \to X$
  that maps $\overline a^1$ to $\overline a^0$.
\end{cor}
\begin{proof}
  Lemma~\ref{lem:back and forth}, Corollary~\ref{cor:main},
  and Theorem~\ref{thm:vaught}.
\end{proof}
\begin{cor}\label{cor:uniqueness}
  Let $\Psi_0' = (\Psi_0', E')$ be a pre-space and $\pi': \Psi_0' \twoheadrightarrow \Psi$
  a continuous surjection inducing $E'$.
  If the image of $\clopen(\Psi_0')$ under $e_{\pi'}$ is an elementary
  substructure of $\regular(\Psi)$,
  then there exists an isomorphism~$\phi: \Psi_0' \to \Psi_0$
  such that $\pi \circ \phi = \pi'$,
  where $\pi: \Psi_0 \twoheadrightarrow \Psi$ is the canonical surjection.
\end{cor}
\begin{proof}
  It is easy to see that an elementary substructure of a cofinally atomic
  model is again cofinally atomic.
  To conclude, recall Corollary~\ref{cor:main} and Theorem~\ref{thm:uniqueness}.
\end{proof}
\bibliographystyle{plain}
\bibliography{thebib}
\end{document}